\documentclass[journal]{IEEEtran}
\usepackage{amsthm}
\usepackage{amsmath,amssymb,amsfonts}
\allowdisplaybreaks
\newcommand\numberthis{\addtocounter{equation}{1}\tag{\theequation}}

\usepackage{graphicx}
\usepackage{graphics}
\usepackage{subfigure}

\usepackage{float}

\usepackage{varioref}
\usepackage{hyperref}
\usepackage[noabbrev]{cleveref}

\usepackage{xcolor}

\usepackage{cite}

\ifCLASSINFOpdf
\else
\fi

\newtheorem{remark}{Remark}
\newtheorem{assumption}{Assumption}

\newtheorem{lemma}{Lemma}
\newtheorem{theorem}{Theorem}

\begin{document}
\title{Smooth Dynamics for Distributed Constrained Optimization with Heterogeneous Delays}
\author{Mengmou~Li, Shunya~Yamashita, Takeshi~Hatanaka, and Graziano~Chesi
\thanks{M.~Li and G.~Chesi are with the Department
of Electrical and Electronic Engineering, The University of Hong Kong, Hong Kong, China (e-mail: mengmou\_li@hku.hk; chesi@eee.hku.hk).}
\thanks{S.~Yamashita is with School of Engineering, Tokyo Institute of Technology, 2-12-1 S5-16,
Ookayama, Meguro-ku, Tokyo, Japan (e-mail:yamashita.s.ag@hfg.sc.e.titech.ac.jp).}
\thanks{T.~Hatanaka is with Graduate School of Engineering, Osaka University, 2-1 Yamada-oka, Suita, Osaka, Japan (e-mail:
hatanaka@eei.eng.osaka-u.ac.jp).}
}

\date{}
\maketitle
\begin{abstract}
This work investigates the distributed constrained optimization problem under inter-agent communication delays from the perspective of passivity. 
First, we propose a continuous-time algorithm for distributed constrained optimization with general convex objective functions. The asymptotic stability under general convexity is guaranteed by the phase lead compensation. The inequality constraints are handled by adopting a projection-free generalized Lagrangian, whose primal-dual gradient dynamics preserves passivity and smoothness, enabling the application of the LaSalle's invariance principle in the presence of delays. Then, we incorporate the scattering transformation into the proposed algorithm to enhance the robustness against unknown and heterogeneous communication delays. Finally, a numerical example of a matching problem is provided to illustrate the results.
\end{abstract}
\section{Introduction}
\IEEEPARstart{D}{istributed} convex optimization over multi-agent systems aims to drive agents to cooperatively optimize the sum of local objective functions that are only accessible to their local agents.
Ever since the pioneer work \cite{wang2011control} that provides a control-theoretic perspective for the proportional-integral (PI) consensus-based distributed algorithms, many works have been carried out in the continuous-time scheme \cite{yang2019survey}.
Recently, some problems in distributed optimization have been analyzed via passivity-based techniques\cite{hatanaka2018passivity,miyano2020continuous,yamashita2020passivity,li2019input,su2019distributed}.
Passivity-based techniques usually enjoy good scalability to large-scale networks owing to the preservation of passivity in parallel or negative feedback interconnection of passive components \cite{hatanaka2015passivity}. 

Distributed optimization in the presence of communication delays has been widely studied in recent years \cite{yang2016distributed,hatanaka2018passivity,miyano2020continuous}.
The work \cite{yang2016distributed} addresses time-varying delays, but it only considers an identical delay known in advance for all communication channels and does not treat inequality constraints, which simplifies convergence analysis.
The problem under unknown and heterogeneous communication delays is addressed via passivity techniques in \cite{hatanaka2018passivity,miyano2020continuous}. However, to ensure optimality in the presence of inequality constraints, delays are assumed to be homogeneous and an additional assumption on the graph is needed in \cite{hatanaka2018passivity}, which is not always easy to verify in large scale networks.
Besides, the objective function is assumed to be strictly convex in \cite{hatanaka2018passivity,miyano2020continuous}, which does not hold for a large class of convex optimization problems.

The LaSalle's invariance principle is widely used for convergence analysis of distributed algorithms.
Algorithms derived from the classical Lagrange multiplier method usually adopt projected operations to guarantee the non-negativeness of the multipliers for inequality constraints. As a result, it leads to non-smooth dynamics, which is analyzed by the invariance principle for Carath{\'e}odory systems \cite{cherukuri2016asymptotic}. 
However, the discontinuous nature hinders the application of the invariance principle when delays are introduced into the systems, which results in the additional restrictive assumptions in \cite{hatanaka2018passivity,miyano2020continuous}. It is worth noting that a projection-free Lagrangian is adopted to solve local and couple inequalities in \cite{li2019generalized}, which enables a smooth dynamics and the application of the LaSalle's invariance principle under delays.
Another important issue is that the primal-dual gradient dynamics may cause oscillations when the objective function lacks strict convexity \cite{holding2017emergence}. To cope with this problem, some modification methods are introduced \cite{holding2017emergence,feijer2010stability}. However, these methods are either restricted to affine constraints or not in a distributed structure.
Recently, a phase lead compensation technique is adopted as a generalized method to ensure convergence \cite{yamashita2020passivity}.

In this work, we address unknown and heterogeneous inter-agent communication delays in distributed constrained optimization without the strictly convex assumption by combining techniques used in \cite{li2019generalized,yamashita2020passivity} from the perspective of passivity.
First, we propose a smooth continuous-time algorithm for distributed constrained optimization with general convex objective functions without delays. Then, we incorporate the scattering transformation into the proposed algorithm to enhance the robustness against unknown and heterogeneous communication delays.

\section{Preliminaries}\label{Preliminaries}
Notations: Let $\mathbb{R}$ ($\mathbb{R}_{\geq 0}$) be the set of (non-negative) real numbers. $\textrm{col}(v_1,\ldots, v_m) := (v_1^T, \ldots, v_m^T)^T$ denotes the column vector stacked with vectors $v_1, \ldots, v_m$. $I_{n}$ denotes the $n \times n$ identity matrix, $\mathbf{1}_n := \textrm{col}(1,\ldots,1) \in \mathbb{R}^n$, and $\mathbf{0}$ denotes the zero matrix of proper dimension. The notation ``$\circ$'' denotes the Hadamard product and ``$\otimes$'' denotes the Kronecker product.
$\nabla_{k}f$ denotes the gradient of $f$ along the variable $k$, whose subscript can be omitted if there is only one variable.

We first introduce some knowledge of convex analysis.
A differentiable function $f: \mathbb{R}^{m} \rightarrow \mathbb{R}$ is \textit{convex} over a convex set $\mathcal{X} \subset \mathbb{R}^{m}$ iff $\left[\nabla f(x)-\nabla f(y)\right]^{T}(x-y)\geq 0$, $\forall x,~y \in \mathcal{X}$, and is \textit{strictly convex} iff the strict inequality holds for any $x \neq y$. The function $f$ is said to be \textit{concave} if $-f$ is convex.
For a function $\mathcal{L}: \mathcal{X} \times \mathcal{Y} \rightarrow \mathbb{R}$ with $\mathcal{X} \subset \mathbb{R}^{n}$, $\mathcal{Y} \subset \mathbb{R}^{m}$ being closed and convex,
$(x^*,y^*) \in \mathcal{X} \times \mathcal{Y}$ is called a \textit{saddle point} of $\mathcal{L}$ if $\mathcal{L}(x^*,y) \leq \mathcal{L}(x^*,y^*)\leq \mathcal{L}(x,y^*)$, $\forall (x,y) \in \mathcal{X} \times \mathcal{Y}$.

Next, let us present some basic concepts in graph theory.
An \textit{undirected} communication graph is represented by $\mathcal{G} = (\mathcal{N},\mathcal{E})$, where $\mathcal{N} = \{1,\ldots,N\}$ is the node set of all agents, $\mathcal{E}\subset \mathcal{N}\times\mathcal{N}$ is the edge set. The edge $(i,j) \in \mathcal{E}$ means that agent $i$ and $j$ can exchange information.
The adjacency matrix $\mathcal{A} := [a_{ij}]$ satisfies $a_{ii} = 0$, and $a_{ij} = a_{ji}> 0$ if $(i,j) \in \mathcal{E}$ and $a_{ij} = 0$, otherwise. The graph $\mathcal{G}$ is said to be \textit{connected} if there exists a sequence of successive edges between any two agents.
When $\mathcal{G}$ is connected and undirected, its corresponding Laplacian matrix $L := \textrm{diag}\{\mathcal{A} \cdot \mathbf{1}_{N}\}- \mathcal{A}$ is positive semidefinite and has zero as its simple eigenvalue associated with eigenvector $v = \alpha \mathbf{1}_{N}$, $\forall \alpha \in \mathbb{R}$.

We conclude this section by giving the definition of passivity \cite{yamashita2020passivity}.
Consider a system $\Sigma$ described by a state model with state $x \in \mathbb{R}^{m}$, input $u \in \mathbb{R}^{n}$ and output $y \in \mathbb{R}^{n}$. The system $\Sigma$ is said to be \textit{passive} if there exists a positive semidefinite differentiable function $S(x): \mathbb{R}^{m} \rightarrow \mathbb{R}_{\geq 0}$ called \textit{storage function}, such that 
$
\dot{S}(x) \leq y^T u
$
holds for all inputs $u(t)$, all initial states $x(0)$, and all $t \geq 0$.

\section{Passivity-Based Algorithm on Constrained Distributed Optimization}
Let us consider a constrained distributed optimization problem in a network of $N$ agents in the node set $\mathcal{N} = \{1,\ldots,N\}$
\begin{equation}\label{distributed problem}
\begin{aligned}
& \min_{z} \sum_{i \in \mathcal{N}} f_i (z),~\text{s.t.}~ g_i(z) \leq 0, ~ h_i(z) = 0, i \in \mathcal{N}
\end{aligned}
\end{equation}
where $z \in \mathbb{R}^{n}$ is a decision variable, $f_i : \mathbb{R}^{n} \rightarrow \mathbb{R}$, $g_i : \mathbb{R}^{n} \rightarrow \mathbb{R}$, $h_i : \mathbb{R}^{n} \rightarrow \mathbb{R}$ are local objective function, inequality constraint and affine equality constraint for the $i$th agent, respectively.
Just for simplicity, we only consider one local inequality and equality constraint for each agent, while it is trivial to extend subsequent results to the case with multiple local constraints.
Next, we adopt the following assumptions.
\begin{assumption}\label{Assumption of optimization problem}
The functions $f_i$ and $g_i$ are convex and twice differentiable. The Slater's condition holds and there exists a finite optimal solution to problem \eqref{distributed problem}.
\end{assumption}
This assumption ensures that the problem is well-defined.
$f_i$ is only required to be convex, implying that there may exist more than one optimal solution to problem \eqref{distributed problem}.

\begin{assumption}\label{Assumpton of graph}
The communication graph $\mathcal{G}$ is undirected and connected.
\end{assumption}

Denote $x = \textrm{col} \left( x_1,\ldots,x_N \right)$, where $x_i \in \mathbb{R}^{n}$, then problem \eqref{distributed problem} is equivalent to
\begin{equation}\label{distributed problem equivalent}
\begin{aligned}
& \min_{x} f(x) : = \sum_{i \in \mathcal{N}} f_i (x_i) + x^T \mathbf{L} x\\
&\text{s.t.} ~g_i(x_i) \leq 0, ~ h_i(x_i) = 0,~ i \in \mathcal{N}, ~\mathbf{L} x = \mathbf{0}
\end{aligned}
\end{equation}
where $\mathbf{L}= L \otimes I_{n}$, and $L$ is the Laplacian matrix.

\subsection{Generalized Lagrange Multiplier Method}
In this subsection, we briefly review the generalized Lagrange multiplier method (GLMM) in \cite{li2019generalized} for solving problem \eqref{distributed problem equivalent}.
Define compact variables
$
\lambda = \textrm{col}(\lambda_1,\ldots,\lambda_N)
$, $\mu = \textrm{col}(\mu_1,\ldots,\mu_N)$, $\xi = \textrm{col}(\xi_1,\ldots,\xi_N)$ with $\lambda_i, \mu_i \in \mathbb{R}$ and $\xi_i \in \mathbb{R}^{n}$, $i\in\mathcal{N}$. Adopt a Lagrangian for problem \eqref{distributed problem equivalent},
\[
\begin{array}{rl}
& \mathcal{L}(x, \xi, \lambda, \mu) := \\
& f(x) + \displaystyle\sum_{i \in \mathcal{N}} \lambda_i^2 g_i(x_i) + \displaystyle\sum_{i \in \mathcal{N}} \mu_i h_i(x_i) - \xi^T \mathbf{L} x + \frac{1}{2} x^T \mathbf{L} x \numberthis \label{generalized Lagrangian}
\end{array}
\]
where $\mathcal{L}(x, \xi, \lambda, \mu)$ is a class of the generalized Lagrangian; $\lambda_i^2$ is the generalized multiplier for the inequality constraint $g_i \leq 0$;
$\mu_i$ is the multiplier for the equality constraint $h_i = 0$; $\xi$ is the multiplier for the consensus constraint $\mathbf{L}x = \mathbf{0}$.
Then, by applying the primal-dual gradient flow to $\mathcal{L}(x, \xi, \lambda, \mu)$, we obtain the following projection-free distributed algorithm
\begin{subequations}\label{eq: distributed algorithm without phase compensation}
\begin{align}
\dot{x}_i =& -\nabla_{x_i} \mathcal{L} = -\nabla f_i(x_i) - \lambda_i^2 \nabla g_i(x_i) - \mu_i \nabla h_i(x_i) \nonumber\\
& + \sum_{j \in \mathcal{N}_i} a_{ij} (x_j - x_i) - \sum_{j \in \mathcal{N}_i} a_{ij} (\xi_j - \xi_i) \label{eq: gradient of x}\\
\dot{\xi}_i =& \nabla_{\xi_i} \mathcal{L},~ \dot{\lambda}_i = \nabla_{\lambda_i} \mathcal{L},~\dot{\mu}_i = \nabla_{\mu_i} \mathcal{L}
\end{align}
\end{subequations}
where $a_{ij}$ is the $(i,j)$-th entry of the adjacency matrix.
The above algorithm is said to be projection-free since the non-negativeness of multipliers $\lambda_i^2$ is already guaranteed without any projection operator.

Let $(x^*, \xi^*, \lambda^*, \mu^*) \in \mathcal{H}^*$ denotes an optimal solution of the problem where $\mathcal{H}^*$ is the set satisfying the generalized KKT condition for problem \eqref{distributed problem equivalent} corresponding to $\mathcal{L}(x, \xi, \lambda, \mu)$, i.e.,
\begin{subequations}\label{KKT conditions}
\begin{align}
& \mathbf{L} x^* =\mathbf{0},\label{kkt Lx}\\
& h_i(x_i^*) = 0, \quad g_i(x_i^*) \leq 0, \quad {\lambda_i^*}^2 g_i(x_i^*) = 0, \label{kkt lambda} \\
& \nabla f_i(x_i^*) + {\lambda_i^*}^2 \nabla g_i(x_i^*) + \mu_i^* \nabla h_i(x_i^*) + \sum_{j=1}^{N} a_{ij} (\xi_j^* - \xi_i^*) = \mathbf{0}\label{kkt gradient of x}
\end{align}
\end{subequations}
where the term $\sum_{j \in \mathcal{N}_i} a_{ij} (x_j^* - x_i^*)$ in \eqref{kkt gradient of x} is omitted since \eqref{kkt Lx} implies $x_i^* = x_j^*$, $\forall i,j$. Next, let us give the following lemma derived from \cite{li2019generalized}.
\begin{lemma}
Under \Cref{Assumption of optimization problem}, a fixed point $(x^*, \xi^*, \lambda^*, \mu^*)$ solves problem \eqref{distributed problem equivalent} if and only if it satisfies condition \eqref{KKT conditions}.
\end{lemma}
Then, denote $z^* = x_i^*$, $\forall i$. Obviously, $z^*$ is the optimal solution to problem \eqref{distributed problem}.

\subsection{GLMM With Phase Lead Compensation}
A restriction for the convergence of algorithm \eqref{eq: distributed algorithm without phase compensation} is $f_i$ being strictly convex \cite{li2019generalized}.
When $f_i$ lacks strict convexity, an extra modification is needed. In this subsection, we add the phase lead compensator into the dynamics \eqref{eq: distributed algorithm without phase compensation}, which serves to provide stable zeros and avoid possible oscillations for the algorithm under general convexity \cite{yamashita2020passivity}.

Define $\nu = \textrm{col} (\nu_1, \ldots, \nu_N)$ with $\nu_i \in \mathbb{R}^{n}$ and 
\[
\begin{array}{rl}
\nu_i = & - \nabla f_i(x_i) - \lambda_i^2 \nabla g_i(x_i) - \mu_i \nabla h_i(x_i)\\
& + \sum_{j \in \mathcal{N}_i} a_{ij} (x_j - x_i) - \sum_{j \in \mathcal{N}_i} a_{ij} (\xi_j - \xi_i). \numberthis \label{dynamics of nu}
\end{array}
\]
We add the phase lead compensator to the integrator in the primal gradient dynamics \eqref{eq: gradient of x}, then the dynamics for the $i$th agent is reformulated in the frequency domain as 
\begin{equation} \label{algorithm compensator}
x_i(s) = \left( M_i(s) \cdot I_{n} \right) \nu_i(s)
\end{equation}
where the generalized transfer function $M_i(s)$ is defined by
\begin{equation}\label{compensator matrix}
\begin{aligned}
&M_i(s) = \sum_{k = 1}^{m} \frac{c^i_{k}}{s + b^i_{k}},\\
& b^i_m > \ldots > b^i_2 > b^i_1 = 0, ~ c^i_k > 0, ~ k = 1,\ldots, m \geq 2.
\end{aligned}
\end{equation}
Note that we only apply the phase lead compensator to \eqref{eq: gradient of x}, and \eqref{compensator matrix} is a simplified version of the algorithm in \cite{yamashita2020passivity}.

Then, the overall distributed algorithm becomes
\begin{subequations}\label{eq: overall distributed algorithm}
\begin{align}
\dot{\rho}_k^{i} =& - b_{k}^{i} \rho_{k}^{i} + c_{k}^{i} \nu_i ,~k=1, \ldots, m \label{eq: time domain of M_i 1}\\
x_i = & \sum_{k=1}^{m} \rho_{k}^{i} \label{eq: time domain of M_i 3}\\
\nu_i = & - \nabla f_i(x_i) - \lambda_i^2 \nabla g_i(x_i) - \mu_i \nabla h_i(x_i)\nonumber\\
& + \sum_{j \in \mathcal{N}_i} a_{ij} (x_j - x_i) - \sum_{j \in \mathcal{N}_i} a_{ij} (\xi_j - \xi_i)
\label{eq: dynamics of nu}\\
\dot{\xi}_i = & \sum_{j \in \mathcal{N}_i} a_{ij} (x_j - x_i)\label{eq: dynamics of xi}\\
\dot{\lambda}_i = & 2 \lambda_i g_i(x_i),~ \dot{\mu}_i = h_i(x_i) \label{eq: dynamics of lambda and mu}
\end{align}
\end{subequations}
where \eqref{eq: time domain of M_i 1}, \eqref{eq: time domain of M_i 3} is the state-space representation of \eqref{algorithm compensator} and $\rho^i_k \in \mathbb{R}^{n}$ ($k = 1, \ldots, m$) is an auxiliary state variable.
Under the phase lead compensation, the block diagram of the $i$th agent's dynamics for $x_i$, $\lambda_i$, $\nu_i$ can be described by \Cref{fig: block diagram of phase lead compensator}.
We can observe that, algorithm \eqref{eq: overall distributed algorithm} is reduced to algorithm \eqref{eq: distributed algorithm without phase compensation} if $M_i(s)$ is replaced by an integrator.

\begin{figure}
\centering
\includegraphics[width = 1\linewidth]{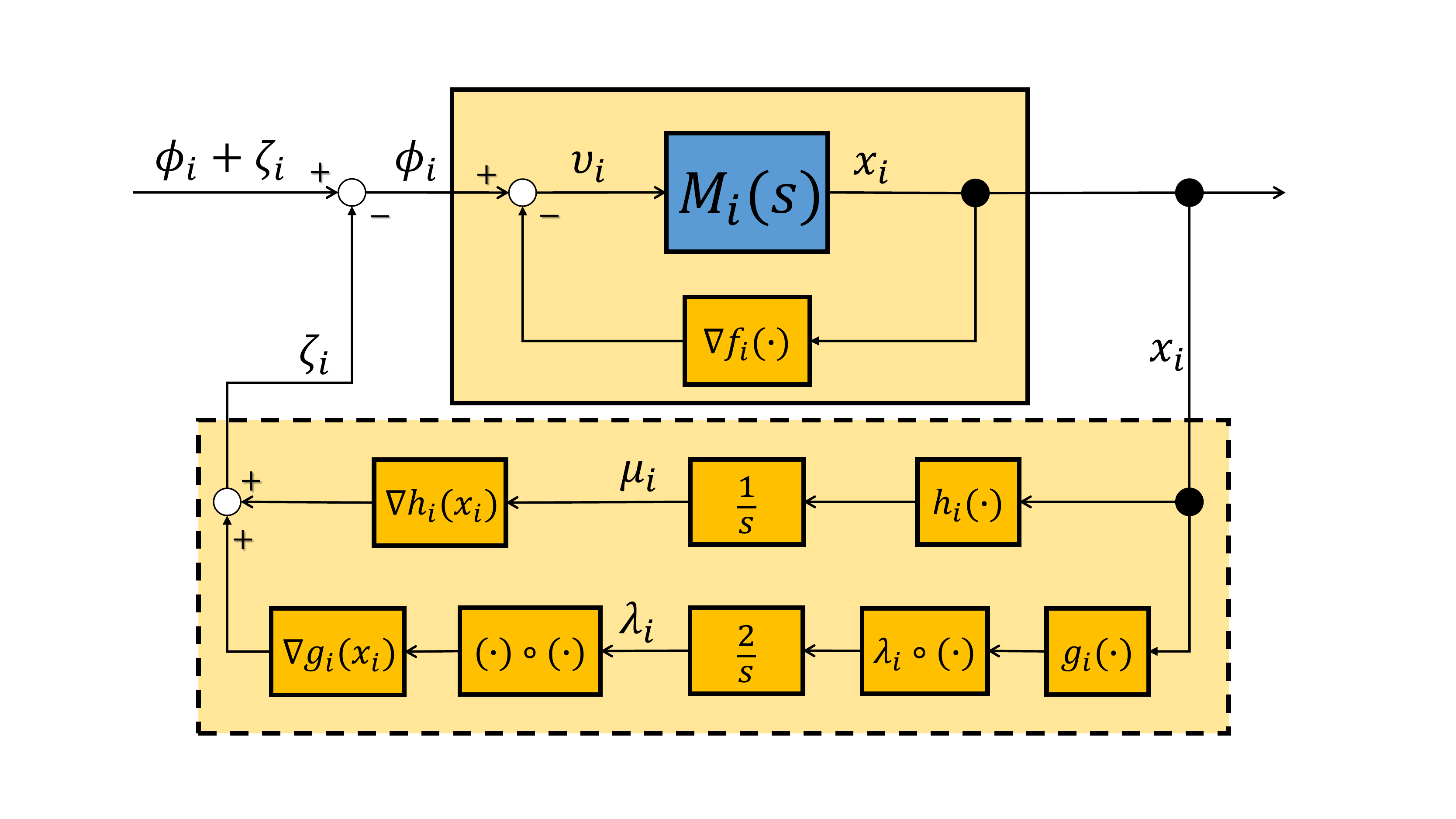}
\caption{Block diagram of the $i$th agent's dynamics for $x_i$, $\lambda_i$, $\nu_i$.
The generalized transfer function $M_i(s)$ represents the cascade connection of phase lead compensator with the integrator for gradient dynamics of $x_i$ to ensure convergence under general convexity.}
\label{fig: block diagram of phase lead compensator}
\end{figure}
\subsection{Convergence Analysis}
We aim to address convergence of system \eqref{eq: overall distributed algorithm} via passivity analysis in this subsection. To this end, let us first analyze the passivity of the subsystems included in it. 
Denote a fixed point $(x^*, \xi^*, \lambda^*, \mu^*) \in \mathcal{H}^*$ as the reference point, $\nu = \textrm{col}\left( \nu_1,\ldots,\nu_N \right)$ and $\nu^* = \nu(x^*, \xi^*, \lambda^*, \mu^*) = \mathbf{0}$. First, we focus on subsystem \eqref{eq: time domain of M_i 1}-\eqref{eq: dynamics of nu} and obtain the following lemma whose proof can be found in \cite[Lemma 7]{yamashita2020passivity}. 
\begin{lemma}[\cite{yamashita2020passivity}]\label{Lemma passivity of compensator + primal dynamics}
Under Assumption \ref{Assumption of optimization problem}, the system defined by \eqref{eq: time domain of M_i 1}-\eqref{eq: dynamics of nu} is passive from $\phi_i - \phi_i^*$ to $x_i - z^*$ with respect to the storage function $S^{c}_{i} = \frac{1}{2c_1^i} \|\rho_1^i - z^* \|^2 + \sum_{k = 2}^{m} \frac{1}{2c_k^i} \| \rho_k^i \|^2$, where $\phi_i := \nu_i + \nabla f_i(x_i)$ and $\phi_i^* := \nu_i^* + \nabla f_i(x_i^*)$.
\end{lemma}
Briefly note that it satisfies that
\[
\dot{S}^{c}_{i} = (x_i - z^*)^T \nu_i - \sum_{k = 2}^{m} \frac{b_k^i}{c_k^i} {\rho_k^i}^T \rho_k^i \leq (x_i - z^*)^T (\phi_i - \phi_i^*). \numberthis \label{eq: time derivative of Sc}
\]

Next, we show that the dual gradient of Lagrangian \eqref{generalized Lagrangian} with respect to $\lambda$, $\mu$ preserves passivity.
\begin{lemma}\label{Lemma passivity of dual dynamics}
Under \Cref{Assumption of optimization problem}, the system given by \eqref{eq: dynamics of lambda and mu} is passive from $x_i - z^*$ to $\zeta_i - \zeta_i^*$ with respect to the storage function $S^{g}_i = \frac{1}{4} \left( \lambda_i^2 - {\lambda_i^*}^2 \right) - \frac{1}{2} {\lambda_i^*}^2 \left(\ln{\lambda_i} - \ln{\lambda_i^*} \right) + \frac{1}{2}\left(\mu_i - \mu_i^* \right)^2$ where $\lambda_i \ln{\lambda_i}$ is defined as $0$ when $\lambda_i = 0$, $\zeta_i := \lambda_i^2 \nabla g_i(x_i) + \mu_i \nabla h_i(x_i)$, and $\zeta_i^* := \zeta_i(z^*, \lambda_i^*, \mu_i^*)$. 
\end{lemma}
\begin{proof}
The storage function $S^{g}_i$ is smooth and differentiable for $\lambda_i$ on $[0,+\infty]$. By direct calculation, $S^{g}_i \geq 0$ and $S^{g}_i = 0$ if and only if $\left(\lambda_i, \mu_i\right) = \left(\lambda_i^*, \mu_i^* \right)$. The time derivative of $S^{g}_i$ gives
\[
\begin{array}{rl}
\dot{S^{g}_i} = &\left( \lambda_i^2 - {\lambda_i^*}^2 \right) g_i(x_i) + \left(\mu_i - \mu_i^* \right) h_i(x_i) \\
    \leq & \lambda_i^2 \left[ g_i(z^*) + \nabla g_i(x_i)^T (x_i - z^*) \right]\\
    	& - {\lambda_i^*}^2 \left[ g_i(z^*) + \nabla g_i(z^*)^T (x_i - z^*) \right]\\
    	& + \left(\mu_i - \mu_i^* \right) \left[\nabla h_i(x_i)^T \left( x_i - z^* \right) \right]\\
    = & [ \lambda_i^2 \nabla g_i(x_i) - {\lambda_i^*}^2 \nabla g_i(z^*)]^T (x_i - z^*) + \lambda_i^2 g_i(z^*)\\
    	& -{\lambda_i^*}^2 g_i(z^*) + \left(\mu_i - \mu_i^* \right) \left[ \nabla h_i(x_i)^T\left( x_i - z^*\right) \right]\\
    \leq & (\zeta_i - \zeta_i^*)^T (x_i - z^*),
\end{array}
\]
where the first inequality follows from the convexity of $g_i$ and affine properties of $h_i$, the second inequality follows from $\nabla h_i(x_i) = \nabla h_i(z^*)$ and the KKT condition \eqref{kkt lambda}.
\end{proof}

We can also observe from \Cref{fig: block diagram of phase lead compensator} that, the system enclosed by the solid line is passive from $\phi_i - \phi_i^*$ to $x_i - z^*$ by \Cref{Lemma passivity of compensator + primal dynamics}. The system within the dashed line is passive from $x_i - z^*$ to $\zeta_i - \zeta_i^*$ by \Cref{Lemma passivity of dual dynamics}.
Moreover, since the communication part in \eqref{eq: dynamics of nu}, \eqref{eq: dynamics of xi} inherits passivity \cite{hatanaka2018passivity}, the overall system \eqref{eq: overall distributed algorithm} can be seen as a feedback interconnection of passive systems. Then, we can obtain the following result on convergence.
\begin{theorem}\label{Theorem 1}
Under Assumptions \ref{Assumption of optimization problem} and \ref{Assumpton of graph}, the trajectories of system \eqref{eq: overall distributed algorithm}
with initial condition $\lambda_i(0) > 0$, $\forall i \in \mathcal{N}$ will converge to a fixed equilibrium point that solves problem \eqref{distributed problem equivalent}.
\end{theorem}

\begin{proof}
Adopt the Lyapunov function candidate $V = \sum_{i \in \mathcal{N}}S^g_i + \sum_{i \in \mathcal{N}} S^c_i + \frac{1}{2} \left\| \xi - \xi^* \right\|^2 \geq 0$. Apparently, $V$ is radially unbounded.
Denote $\zeta = \textrm{col}(\zeta_1, \ldots,\zeta_N)$, $\zeta^* = \textrm{col}(\zeta_1^*, \ldots,\zeta_N^*)$,
by \Cref{Lemma passivity of compensator + primal dynamics,Lemma passivity of dual dynamics}, the time derivative of $V$ satisfies
\begin{align*}
\dot{V} \leq & (x - x^*)^T (\phi - \phi^*) + (\zeta - \zeta^*)^T (x - x^*)\\
& - \left( \xi - \xi^* \right)^T \mathbf{L} \left( x - x^* \right)
\leq  - \left( x - x^*\right)^T \mathbf{L}(x - x^*)
\leq  0,
\end{align*}
where the last inequality follows from the positive semidefiniteness of $\mathbf{L}$.
Then the states are bounded. The set $\Omega_{0} := \left\{(x, \xi, \lambda, \mu, \rho) | V \leq V(0) \right\}$ is a positively invariant set. 
Invoking the LaSalle's invariance principle, the states will converge to the largest invariant set in $\{(x, \xi, \lambda, \mu, \rho) |\dot{V} = 0\}$ which we denote as $\Omega_{c}$ in the subsequent, and $\dot{V} = 0$ only if all the non-negative terms are zero.
Notice that $ - \sum_{k = 2}^{m} \frac{b_k^i}{c_k^i} {\rho^i_k}^T \rho^i_k = 0$ holds only if $\rho_{k}^{i} \equiv \mathbf{0}$, $\forall k = 2, \dots, m$, which implies that $\nu_i = \mathbf{0}$ and thus $x_i$ is unchanged.
$- \left( x - x^*\right)^T \mathbf{L}(x - x^*) = 0$ implies that $\mathbf{L} x = \mathbf{0}$ and $\dot{\xi} = \mathbf{0}$.
Then $x$ satisfies \eqref{kkt gradient of x}.

Next, let us look at the dynamics \eqref{eq: dynamics of lambda and mu} when the states converges to $\Omega_c$. If $\lambda_i^* = 0$, then the constraint $g_i(x_i)$ is inactive, meaning that $g_i(x_i) \leq 0$ when the primal gradient of $\mathcal{L}$ with respect to $x_{i}$ vanishes, i.e., when \eqref{kkt gradient of x} holds. If $\lambda_i^* > 0$, recalling the definition of $S^g_i$, we obtain that $\lambda_i > 0$, for all $t > 0$, because if $\lambda_i \rightarrow 0$, then $S^g_i \rightarrow +\infty$ due to the term $-\frac{1}{2}{\lambda_i^*}^2 \ln{\lambda_i}$, which contradicts the fact that $V$ is decreasing. Since $\lambda_i$ is nonzero, $g_i(x_i)$ should be zero to ensure stability and boundedness of the dynamics $\dot{\lambda}_i = 2\lambda_i g_i$.
Similarly, it can be observed that $h_i(x_i)$ should be zero to ensure boundedness.
Therefore, the KKT condition \eqref{KKT conditions} is satisfied, i.e., $(x, \xi, \lambda, \mu) \in \mathcal{H}^*$ and is unchanged whenever $(x, \xi, \lambda, \mu, \rho)$ converges to $\Omega_{c}$. In conclusion, the trajectories generated by the algorithm will asymptotically converge to a constant equilibrium point that solves problem \eqref{distributed problem equivalent}.
\end{proof}


\section{Constrained Distributed Optimization with Heterogeneous Communication Delays}\label{Section Delays}
Let us consider the presence of unknown and heterogeneous inter-agent communication delays in this section. For $(i, j) \in \mathcal{E}$, let the communication delay from agent $i$ to $j$ be denoted by a constant $T_{ij}$.
In this case, each agent cannot catch the current variable of its neighboring agents. Then the algorithm under delays becomes
\begin{subequations}\label{eq: overall distributed algorithm under delays}
\begin{align*}
\nu_i = & - \nabla f_i(x_i) - \lambda_{i}^2 \nabla g_i(x_i) - \mu_i \nabla h_i(x_i)\nonumber\\
& + \sum_{j \in \mathcal{N}_i} a_{ij} (r_{ij}^{x} - x_i) - \sum_{j \in \mathcal{N}_i} a_{ij} (r_{ij}^{\xi} - \xi_i)\numberthis \label{eq: dynamics of nu under delays}\\
\dot{\xi}_i = & \sum_{j \in \mathcal{N}_i} a_{ij} (r_{ij}^{x} - x_i) \numberthis\label{eq: dynamics of xi under delays}\\
&\eqref{eq: time domain of M_i 1}, \eqref{eq: time domain of M_i 3}, \eqref{eq: dynamics of lambda and mu}
\end{align*}
\end{subequations}
where \eqref{eq: dynamics of nu under delays} and \eqref{eq: dynamics of xi under delays} are the modified update expression against \eqref{eq: dynamics of nu} and \eqref{eq: dynamics of xi} by replacing the neighbor's information with $r_{ij}^{x}$, $r_{ij}^{\xi}$, which denote the information agent $i$ receives from agent $j$.
The dynamics \eqref{eq: dynamics of nu under delays}, \eqref{eq: dynamics of xi under delays} can be rewritten as 
\[ \begin{array}{rl}
\begin{bmatrix}
\nu_i\\
\dot{\xi}_i
\end{bmatrix}
= 
\begin{bmatrix}
- \nabla f_i(x_i) - \zeta_i\\
\mathbf{0}
\end{bmatrix}
+
p_{i} \numberthis \label{eq: rewritten of nu and xi}
\end{array}
\]
where $\zeta_i$ is defined in \Cref{Lemma passivity of dual dynamics},
$
p_{i} := \sum_{j \in \mathcal{N}_{i}} p_{ij}, ~
p_{ij} :=
\begin{bmatrix}
p_{ij}^{x}\\
p_{ij}^{\xi}
\end{bmatrix}
=
E_{ij}
\begin{bmatrix}
r_{ij}^{x} - x_i\\
r_{ij}^{\xi} - \xi_i
\end{bmatrix}
$
and 
$
E_{ij} :=
\begin{bmatrix}
a_{ij} & -a_{ij}\\
a_{ij} & 0
\end{bmatrix} \otimes I_{n},
~ j \in \mathcal{N}_i.
$

Define
$
\bar{x}_i := x_i - z^*, ~ \bar{\xi}_i := \xi_i - 2{\xi_i^*}$,
$
\bar{p}_i := p_i - \sum_{j \in \mathcal{N}_i} p_{ij}^* = \sum_{j \in \mathcal{N}_i} \bar{p}_{ij}, ~ 
p_{ij}^* :=
a_{ij}
\begin{bmatrix}
\xi_{i}^{*} - \xi_{j}^{*}\\
\mathbf{0}
\end{bmatrix}$,
$
r_{ij} :=
\begin{bmatrix}
r_{ij}^{x}\\
r_{ij}^{\xi}
\end{bmatrix},~
\bar{r}_{ij} := \begin{bmatrix}
\bar{r}_{ij}^{x}\\
\bar{r}_{ij}^{\xi}
\end{bmatrix}
=
r_{ij} - r_{ij}^*, ~ 
r_{ij}^* := 
\begin{bmatrix}
z^*\\
\xi_i^* - \xi_j^*
\end{bmatrix}.
$
We show that the dynamics under communication delays preserves passivity-like properties.
\begin{lemma}\label{lemma passivity-like property}
Under \Cref{Assumption of optimization problem}, the system given by \eqref{eq: overall distributed algorithm under delays} has the passivity-like property $\dot{S}_i \leq \sum_{j\in\mathcal{N}_i} \bar{r}_{ij}^T \bar{p}_{ij}$ where $S_{i} = S_{i}^{c} + S_{i}^{g} + \frac{1}{2} \left\| \bar{\xi}_i \right\|^2$, $S_{i}^{c}$ and $S_{i}^{g}$ are defined in \Cref{Lemma passivity of compensator + primal dynamics,Lemma passivity of dual dynamics}.
\end{lemma}
\begin{proof}
From the former lemmas, $S_i \geq 0$. The time derivative of $S_{i}$ satisfies 
\[ \begin{array}{rl}
\dot{S}_i \leq & - \sum_{k = 2}^{m} \frac{b_k^i}{c_k^i} {\rho_k^i}^T \rho_k^i + \bar{x}_i^T \left[- \nabla f_i(x_i) - \zeta_i \right]\\
&
+ \bar{x}_i^T \left(\zeta_i - \zeta_i^*\right) +
\begin{bmatrix}
\bar{x}_i & \bar{\xi}_i
\end{bmatrix}
p_i\\
\leq &
- \bar{x}_i^T \left(\zeta_i - \zeta_i^*\right)
+ \bar{x}_i^T \left(\zeta_i - \zeta_i^*\right)
- \sum_{k = 2}^{m} \frac{b_k^i}{c_k^i} {\rho_k^i}^T \rho_k^i\\
& + \bar{x}_{i}^T \sum_{j \in \mathcal{N}_i} a_{ij} \left( \xi_j^* - \xi_i^*\right)
+ \sum_{j \in \mathcal{N}_i}\begin{bmatrix}
\bar{x}_i &
\bar{\xi}_i
\end{bmatrix}
p_{ij}\numberthis \label{time derivative of S_i}\\
= & - \sum_{k = 2}^{m} \frac{b_k^i}{c_k^i} {\rho_k^i}^T \rho_k^i + \sum_{j \in \mathcal{N}_i}
\begin{bmatrix}
\bar{x}_i + \bar{r}_{ij}^{x} - \bar{r}_{ij}^{x} \\
\bar{\xi}_i + \bar{r}_{ij}^{\xi} - \bar{r}_{ij}^{\xi}
\end{bmatrix}^T
\bar{p}_{ij}\\
\leq & \sum_{j \in \mathcal{N}_i} \bar{r}_{ij}^T \bar{p}_{ij} - \sum_{j\in \mathcal{N}_{i}} a_{ij} \left\| \bar{x} - \bar{r}_{ij}^{x} \right\|^2\\
\leq &
\sum_{j \in \mathcal{N}_i} \bar{r}_{ij}^T \bar{p}_{ij}
\end{array}
\]
where the first inequality follows from \eqref{eq: time derivative of Sc}, \eqref{eq: rewritten of nu and xi} and \Cref{Lemma passivity of dual dynamics}.
\end{proof}

If state variables are exchanged, then agent $i$ at time $t$ receives $r_{ij}^{x} = x_{j}(t - T_{ji})$ and $r_{ij}^{\xi} = \xi_{j}(t - T_{ji})$ from agent $j$ due to the existence of delays. However, this may cause instability and divergence to the dynamics \cite{hatanaka2018passivity}. Thus, we do not directly exchange original state variables here.
To ensure stability under delays, a scattering transformation method is introduced \cite{anderson1989bilateral,hatanaka2015passivity}.
The scattering transformation in this work is defined as
\begin{subequations}\label{scattering variables}
\begin{align}
s_{\overrightarrow{ij}} = & \frac{1}{\sqrt{2\eta}}(- p_{ij} + \eta r_{ij}),  & s_{\overleftarrow{ij}} = \frac{1}{\sqrt{2\eta}}(p_{ij} + \eta r_{ij})\\
s_{\overleftarrow{ji}} = & \frac{1}{\sqrt{2\eta}}(p_{ji} + \eta r_{ji}), & s_{\overrightarrow{ji}} = \frac{1}{\sqrt{2\eta}}(- p_{ji} + \eta r_{ji})
\end{align}
\end{subequations}
for $(i, j) \in \mathcal{E}$, where $\eta > 0$.
Specifically, $s_{\overrightarrow{ij}}$ denotes the signal that agent $i$ sends to agent $j$ while $s_{\overleftarrow{ji}}$ represents the signal $j$ receives from $i$. The other notations are defined similarly. Due to the delays, these signals should satisfy
\begin{equation}\label{scattering variables relation with delays}
    s_{\overleftarrow{ji}}(t) = s_{\overrightarrow{ij}}(t - T_{ij}), ~~ s_{\overleftarrow{ij}}(t) = s_{\overrightarrow{ji}}(t - T_{ji}).
\end{equation}
Instead of directly exchanging $x$ and $\xi$, scattering variables \eqref{scattering variables} are exchanged between agent $i$ and $j$ for $(i, j) \in \mathcal{E}$. Then the input variables $r_{ij}^{x}$, $r_{ij}^{\xi}$ for each agent are computed from these scattering variables. For simplicity, we suppose that $s_{\overleftarrow{ij}}(t) = s_{\overrightarrow{ij}}(t) = 0$, $\forall t < 0$.

It has been proved that the system of the scattering transformation inherits passivity properties. 
\begin{lemma}[\cite{hatanaka2018passivity}]\label{lemma passivity of scattering transformation}
The system consisting of \eqref{scattering variables} and \eqref{scattering variables relation with delays} is passive from $-[\bar{p}_{ij}^T, ~ \bar{p}_{ji}^T]^T$ to $[\bar{r}_{ij}^T, ~ \bar{r}_{ji}^T]^T$ with respect to the storage function
\[
\begin{array}{rl}
V_{ij}(t) = \frac{1}{2}
\displaystyle \int_{0}^{t}
& \left( \left\|s_{\overrightarrow{ij}}(\tau) + \gamma_{ij}^* \right\|^2 - \left\|s_{\overleftarrow{ji}}(\tau) + \gamma_{ij}^* \right\|^2 \right.\\
& \left. + \left\|s_{\overrightarrow{ji}}(\tau) - \delta_{ij}^* \right\|^2 - \left\|s_{\overleftarrow{ij}}(\tau) - \delta_{ij}^* \right\|^2 \right)d \tau\\
& + \frac{T_{ij}}{2}(\gamma_{ij}^*)^2 + \frac{T_{ji}}{2} (\delta_{ij}^*)^2
\end{array}
\]
where $\gamma_{ij}^* := \frac{1}{\sqrt{2\eta}} \left( p_{ij}^* - \eta r_{ij}^* \right)$, $\delta_{ij}^* := \frac{1}{\sqrt{2\eta}} \left( p_{ij}^* + \eta r_{ij}^* \right)$.
\end{lemma}
Since $\bar{p}_{ij} = p_{ij} - p_{ij}^*$, we can easily obtain the following lemma by taking $p_{ij}^* = p_{ji}^* =  \mathbf{0}$ in the proof of \Cref{lemma passivity of scattering transformation}.
\begin{lemma}\label{lemma passivity of scattering transformation 2}
The system consisting of \eqref{scattering variables} and \eqref{scattering variables relation with delays} is passive from $-[p_{ij}^T, ~ p_{ji}^T]^T$ to $[\bar{r}_{ij}^T, ~ \bar{r}_{ji}^T]^T$.
\end{lemma}

Following the above lemmas, the algorithm with scattering transformation controllers can be viewed as a feedback interconnection of passive systems and hence preserves passivity.
Then, we can obtain the convergence result.
\begin{theorem}\label{Theorem 2}
Under Assumptions \ref{Assumption of optimization problem} and \ref{Assumpton of graph}, the trajectories of system \eqref{eq: overall distributed algorithm under delays}
with controller \eqref{scattering variables}, \eqref{scattering variables relation with delays} and initial condition $\lambda_i(0) > 0$, $\forall i \in \mathcal{N}$ will asymptotically converge to a fixed equilibrium that solves problem \eqref{distributed problem equivalent}.
\end{theorem}
\begin{proof}
\textbf{Step 1}: We adopt the Lyapunov function candidate $\bar{V} = \sum_{i \in \mathcal{N}} S_i + \sum_{(i,j) \in \mathcal{E}} V_{ij}$ where $S_i$ and $V_{ij}$ are defined in \Cref{lemma passivity-like property} and \Cref{lemma passivity of scattering transformation}, respectively. Obviously, $\bar{V} \geq 0$ and is radially unbounded. By following \Cref{lemma passivity-like property} and \Cref{lemma passivity of scattering transformation}, the time derivative of $\bar{V}$ satisfies
\begin{equation}\label{time derivative of bar_V}
\dot{\bar{V}} \leq \sum_{i \in \mathcal{N}} \left\{ - \sum_{k = 2}^{m} \frac{b_k^i}{c_k^i} {\rho_k^i}^T\rho_k^i - \sum_{j\in \mathcal{N}_{i}} a_{ij} \left\| \bar{x} - \bar{r}_{ij}^{x} \right\|^2 \right\}
\leq 0.
\end{equation}
Then the system states are bounded.
The set $\bar{\Omega}_{0} : = \{(x, \xi, \lambda, \mu, \rho) | \bar{V} \leq \bar{V}(0) \}$ is a positively invariant set.
By the LaSalle's invariance principle for delay systems \cite[Theorem 5.17]{smith2011introduction}, the states will converge to the largest invariant set $\bar{\Omega}_{c}$ in $\{ (x, \xi, \lambda, \mu, \rho) |\dot{\bar{V}} = 0 \}$, which implies that $\dot{\xi}_i = \mathbf{0}$, $x_i = r_{ij}^{x}$. Moreover, $\nu_i = \mathbf{0}$ implies that $x_i$ remains unchanged.

However, these results derived from \eqref{time derivative of bar_V} are insufficient to conclude the optimality yet. To this end, let us go back and rearrange the time derivative of $\dot{\bar{V}}$.

\textbf{Step 2}:
Reformulating the term
$\sum_{j \in \mathcal{N}_i}\begin{bmatrix}
\bar{x}_i &
\bar{\xi}_i
\end{bmatrix}
p_{ij}$ from \eqref{time derivative of S_i}, we have
\[
    \begin{array}{rl}
& \sum_{j \in \mathcal{N}_i}\begin{bmatrix}
\bar{x}_i & \bar{\xi}_i
\end{bmatrix}
p_{ij}\\
 = &
\sum_{j \in \mathcal{N}_i}
\left(
\begin{bmatrix}
\bar{r}_{ij}^{x} & 
\bar{r}_{ij}^{\xi}
\end{bmatrix}
- 
\begin{bmatrix}
\bar{r}_{ij}^{x} - \bar{x}_i &
\bar{r}_{ij}^{\xi} - \bar{\xi}_i
\end{bmatrix}
\right)
p_{ij} \\
= & \sum_{j \in \mathcal{N}_{i}}\bar{r}_{ij}^T p_{ij}
- \sum_{j\in \mathcal{N}_{i}}
\begin{bmatrix}
r_{ij}^{x} - x_i\\
r_{ij}^{\xi} -\xi_i
\end{bmatrix}^T
E_{ij}
\begin{bmatrix}
r_{ij}^{x} - x_i\\
r_{ij}^{\xi} -\xi_i
\end{bmatrix}\\
& -
\sum_{j \in \mathcal{N}_i}
\begin{bmatrix}
\mathbf{0}\\
\xi_i^* - \xi^*_{j}
\end{bmatrix}^T
E_{ij}
\begin{bmatrix}
r_{ij}^{x} - x_i\\
r_{ij}^{\xi} -\xi_i
\end{bmatrix}\\
= &
\sum_{j \in \mathcal{N}_{i}}\bar{r}_{ij}^T p_{ij}
- \sum_{j\in \mathcal{N}_{i}} a_{ij} \left( \xi_i^* - \xi_j^* \right)^T \left(r_{ij}^{x} - x_i\right)\\
& - 
\sum_{j\in \mathcal{N}_{i}} a_{ij} \left\| x_i - r_{ij}^{x} \right\|^2.
\end{array}
\]
When the states are in $\bar{\Omega}_{c}$, it is already shown that $x_i = r_{ij}^{x} = \text{constant}$. Thus, following the time derivative of $\bar{V}$ along with \Cref{lemma passivity of scattering transformation 2}, we have
\begin{align*}
\dot{\bar{V}} = & \sum_{i \in \mathcal{N}} \left\{ - \sum_{k = 2}{m} {\rho_k^i}^T\rho_k^i + \left( \mu_i - \mu_i^* \right) h_i(x_i) \right.\\
& + (\lambda_i^2 - {\lambda_i^*}^2) g_i(x_i) - \bar{x}_i^T \left[\nabla f_i(x_i) + \lambda_i^2 \nabla g_i(x_i) \right]\\
&\left. - \sum_{j\in \mathcal{N}_{i}} a_{ij} \left\| x_i - r_{ij}^{x} \right\|^2 \right\}\\
= & \sum_{i \in \mathcal{N}} \left\{ - \bar{x}_i^T \left[ \nabla f_i(x_i) + \lambda_i^2 \nabla g_i(x_i) \right] + (\lambda_i^2 - {\lambda_i^*}^2) g_i(x_i) \right.\\
& \left. + (\mu_i - \mu_i^*) h_i(x_i) \vphantom{{\lambda_i^*}^2} \right\}\\
= & - \bar{x}^T\nabla_{x} \mathcal{L}_{g} + \left(\lambda^2 - {\lambda^*}^2 \right)^T \nabla_{\lambda^2} \mathcal{L}_{g} + \bar{\mu}^T \nabla_{\mu}\mathcal{L}_{g}\\
\leq & \mathcal{L}_{g}(x^*, \lambda, \mu) - \mathcal{L}_{g}(x,\lambda,\mu) + \mathcal{L}_{g}(x, \lambda, \mu) - \mathcal{L}_{g}(x,\lambda^*, \mu^*)\\
= & \mathcal{L}_{g}(x^*, \lambda, \mu) - \mathcal{L}_{g}(x^*,\lambda^*, \mu^*)\\
& + \mathcal{L}_{g}(x^*, \lambda^*, \mu^*) - \mathcal{L}_{g}(x,\lambda^*, \mu^*)
\end{align*}
where $\lambda^2 := \lambda \circ \lambda$ and ${\lambda^*}^2$ is defined similarly, $\mathcal{L}_{g}(x, \lambda, \mu) := f(x) + \sum_{i \in \mathcal{N}} \left[\lambda_i^2 g_i(x_i) + \mu_i h_i(x_i) \right]$ is a Lagrangian, the third equality follows from the fact that $\mathcal{L}_{g}$ is convex with respect to $x$ and concave with respect to $\left( \lambda^2, \mu \right)$. Then $\left(x^*, {\lambda^*}^2, \mu^* \right)$ can be seen as a saddle point to $\mathcal{L}_{g}$. It satisfies that $\mathcal{L}_{g}(x^*, \lambda, \mu) \leq \mathcal{L}_{g}(x^*,\lambda^*, \mu^*) \leq \mathcal{L}_{g}(x,\lambda^*, \mu^*)$. These equalities hold when $\dot{\bar{V}} = 0$, i.e.,
\begin{subequations}
\begin{align}
& \sum_{i\in\mathcal{N}} \left(\lambda_i^2 - {\lambda_i^*}^2\right)g_i(z^*) + \left(\mu_i - \mu_i^* \right) h_i(z^*)= 0\\
& \sum_{i \in \mathcal{N}} f_i (z^*) = \sum_{i \in \mathcal{N}} \left\{ f_i(x_i) + {\lambda_i^*}^2 g_i(x_i) + \mu_i^* h_i(x_i) \right\}.\numberthis \label{optimal value comparing}
\end{align}
\end{subequations}
Here, since $x_i$ is unchanged, it is clear from \eqref{eq: dynamics of lambda and mu} that $h_i(x_i) = 0$ otherwise $\mu_i$ is unbounded, $\forall i \in \mathcal{N}$, then $\sum_{i \in \mathcal{N}} \mu_i^* h_i(x_i) = 0$.
Therefore, if
$
{\lambda_i^*}^2 g_i(x_i) = 0,~\forall i \in \mathcal{N},
$
then we can conclude from \eqref{optimal value comparing} that $x_i$ is the optimal solution. If $\lambda^* \neq 0$, then $\lambda_i > 0$, $\forall t > 0$. This is because if $\lambda_i$ goes to zero, then $\bar{V} \rightarrow + \infty$ due to the term $-\frac{1}{2} {\lambda_i^*}^2 \ln{\lambda_i}$ in $S_{i}^{g}$, which contradicts the fact that $\bar{V}$ is decreasing.
We will reason by cases in the following.
\begin{enumerate}
    \item If $g_i(z^*) < 0$, then $\lambda_i^* = 0$, ${\lambda_i^*}^2 g_i(x_i) = 0$ holds.
    \item If $g_i(z^*) = 0$, then $\lambda_i^*$ can be nonzero. Note that $x_i$ is a constant. Then $g_i(x_i) \leq 0$ or else $\lambda_i$ will be unbounded according to the dynamics $\dot{\lambda}_i = 2 \lambda g_i(x_i)$.
    \begin{enumerate}
        \item If $g_i(x_i) = 0$, ${\lambda_i^*}^2 g_i(x_i) = 0$ holds.
        \item If $g_i(x_i) < 0$, then $\lambda_i$ should be zero to render a stable equilibrium point, which contradicts the fact that $\lambda_i$ will not approach $0$, $\forall t > 0$.
    \end{enumerate}
\end{enumerate}
Therefore, ${\lambda_i^*}^2 g_i(x_i) = 0$ holds, and $\sum_{i\in\mathcal{N}} f_i(x_i)$ is the optimal value, which means that $x_i = x_j,~\forall i, ~ j$ is an optimal solution.
\end{proof}

\begin{remark}
The LaSalle's invariance principle plays a crucial role, which allows the analysis in step 2 of the proof. Such an application of the LaSalle's invariance principle under delays is made valid thanks to the algorithmic dynamics \eqref{eq: distributed algorithm without phase compensation} for the generalized Lagrangian that preserves smoothness and passivity.
It should also be noted that the passivity-based phase lead compensation technique eliminates possible oscillations and ensures the convergence with cost functions not necessarily strictly convex.
\end{remark}

\section{Application to Target Matching Problem}
Let us consider an environmental-monitoring problem that is formulated as a target matching problem \cite{miyano2019design}.
\begin{equation}\label{matching problem original}
\begin{aligned}
&\min_{z_{lk} \geq 0} \sum_{l = 1}^{N} \sum_{k = 1}^{M} z_{lk} \left\|w_l - q_k \right\|\\
&\text{s.t.}~\sum_{l = 1}^{N} z_{lk} = 1, ~ k = 1, \ldots, M,~ \sum_{k = 1}^{M} z_{lk} = 1, ~ l = 1, \ldots, N
\end{aligned}
\end{equation}
where $z = \textrm{col}(z_1, \ldots, z_N) \in \mathbb{R}^{NM}$ is the decision variable with $z_l = \textrm{col}(z_{l1}, \ldots, z_{lM})\in \mathbb{R}^{M}$, and $N,~M$ are the number of Robots and Targets, respectively. The variable $z_{lk} \in \mathbb{R}$ denotes the matching label for Robot $l$ and Target $k$. The term $\left\|w_l - q_k \right\|$ denotes the distance from Robot $l$ to Target $k$, which is regarded as a constant (given by sensing). Note that the linear programming problem \eqref{matching problem original} with continuous variables is a strictly relaxation from integer programming. The optimal solution for $z_{lk}$ is either $1$ or $0$, which implies the matching status between Robot $l$ and Target $k$ \cite{miyano2019design}.
We reformulate \eqref{matching problem original} as a consensus-based distributed optimization problem and use notation $x_{i(lk)}$ to denote the estimation of $z_{lk}$ from agent $i$.

Consider an area of $100 \times 100 [\text{m}^2]$ with $N = 5$ Robots and $M = 5$ Targets. The positions of Robots and Targets are shown in \Cref{fig: Robots and Targets positions}.

\begin{figure}[b]
\centering
\subfigure[Robots and Targets positions.]{
\includegraphics[width = .23\textwidth]{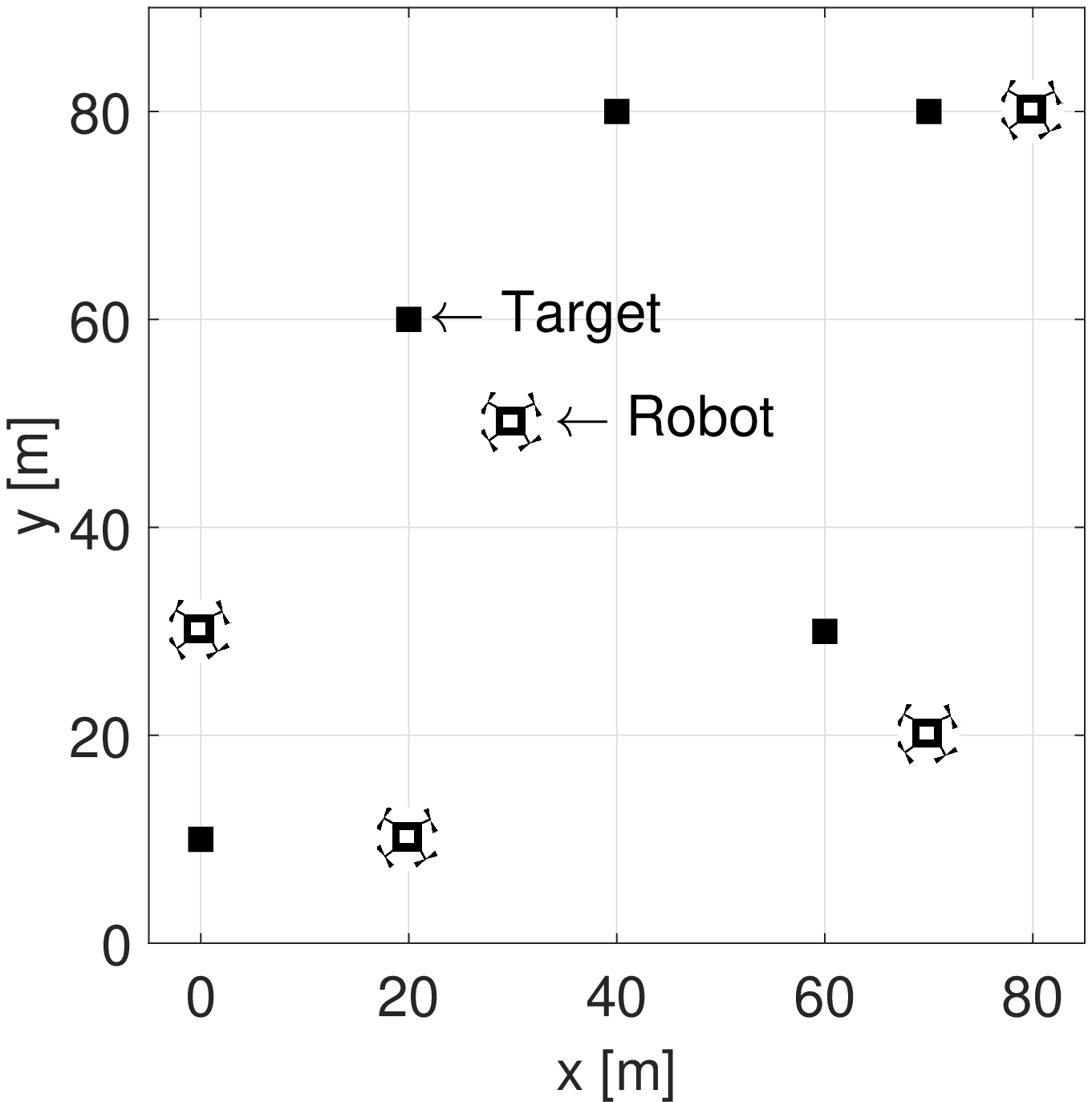}\label{fig: Robots and Targets positions}}
\subfigure[Matching results.]{
\includegraphics[width = .23\textwidth]{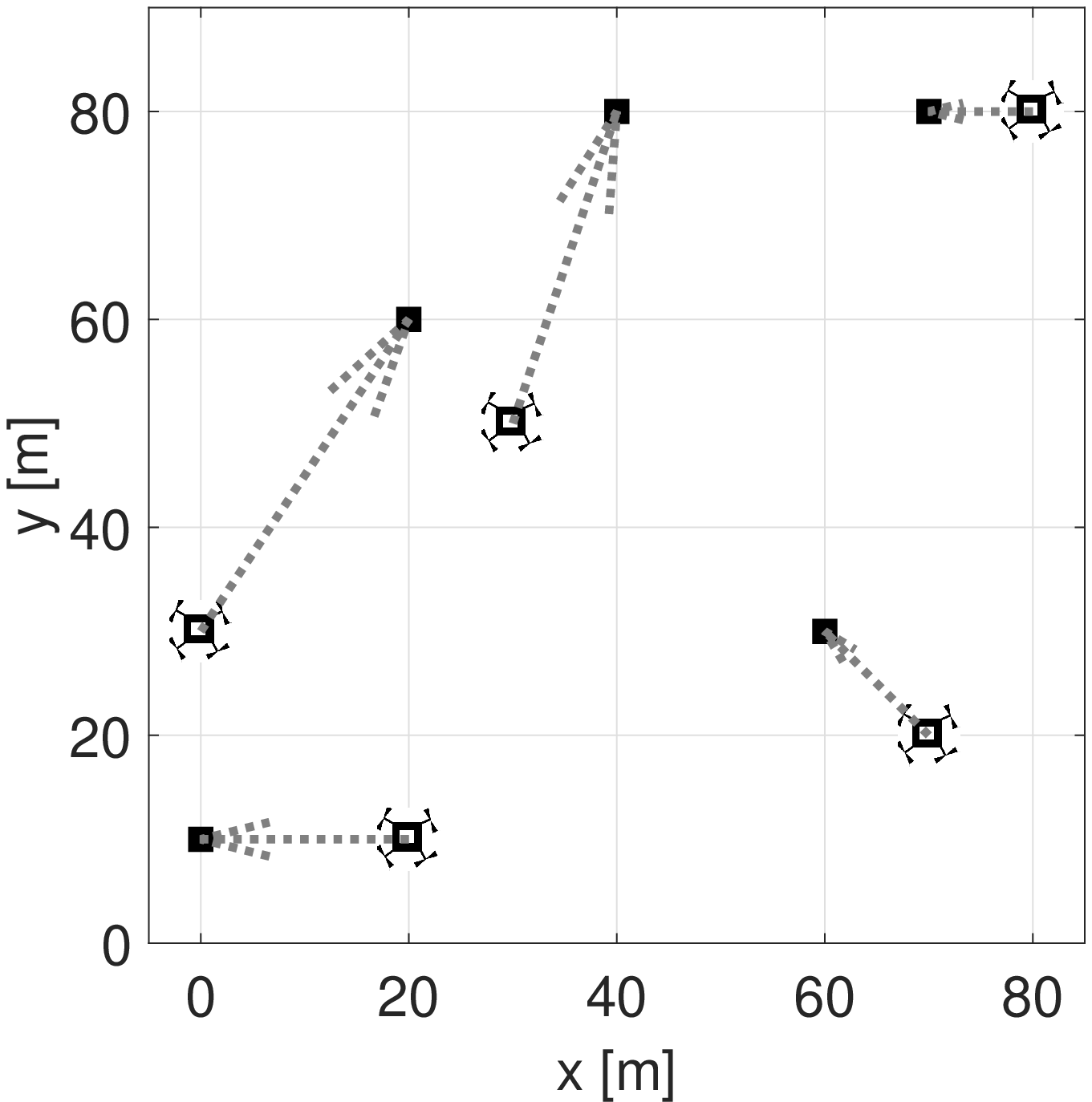}
\label{fig: Matching Results Robots and Targets}}
\caption{Robots and Targets in the $100 \times 100[\text{m}^2]$ area.}
\end{figure}

The communication graph is set to be a ring graph with $a_{ij} = 4$, $\forall (i,j) \in \mathcal{E}$, the other parameters are defined as $m = 2$, $b_2^i = 5$, $c_1^i = 1$, $c_2^i = 10$, $\eta = 1$, the initial condition is $(x_i,\xi_i, \lambda_i, \mu_i, \rho^i)|_{t = 0} = (\mathbf{0}, \mathbf{0}, 0.01, 0, \mathbf{0})$ and the stepsize is set to $0.001$ in Simulink. The trajectories of of $x_i$, $i = 1,\ldots,5$ in algorithm \eqref{eq: overall distributed algorithm} without communication delay converge to the optimal solution, as shown in \Cref{fig: trajectories no delay}, which validates \Cref{Theorem 1}.

Next, we assume that there exist unknown and heterogeneous constant delays between any two neighboring agents. The communication delays can bring in instability for algorithm \eqref{eq: overall distributed algorithm} (see, e.g., \cite{hatanaka2018passivity}).
Thus, let us adopt algorithm \eqref{eq: overall distributed algorithm under delays} with scattering transformation \eqref{scattering variables} to enhance robustness against delays, and consider heterogeneous communication delays $T_{ij} \in \left[0.2, 0.3\right][\text{s}]$, $\forall (i,j) \in \mathcal{E}$.
The trajectories of $x_i$, $i = 1,\ldots,5$ converge to the optimal solution as shown in \Cref{fig: trajectories}, validating \Cref{Theorem 2}.
\begin{figure}
\centering
\subfigure[The algorithm \eqref{eq: overall distributed algorithm} without communication delay.]{
\includegraphics[width = .48\linewidth]{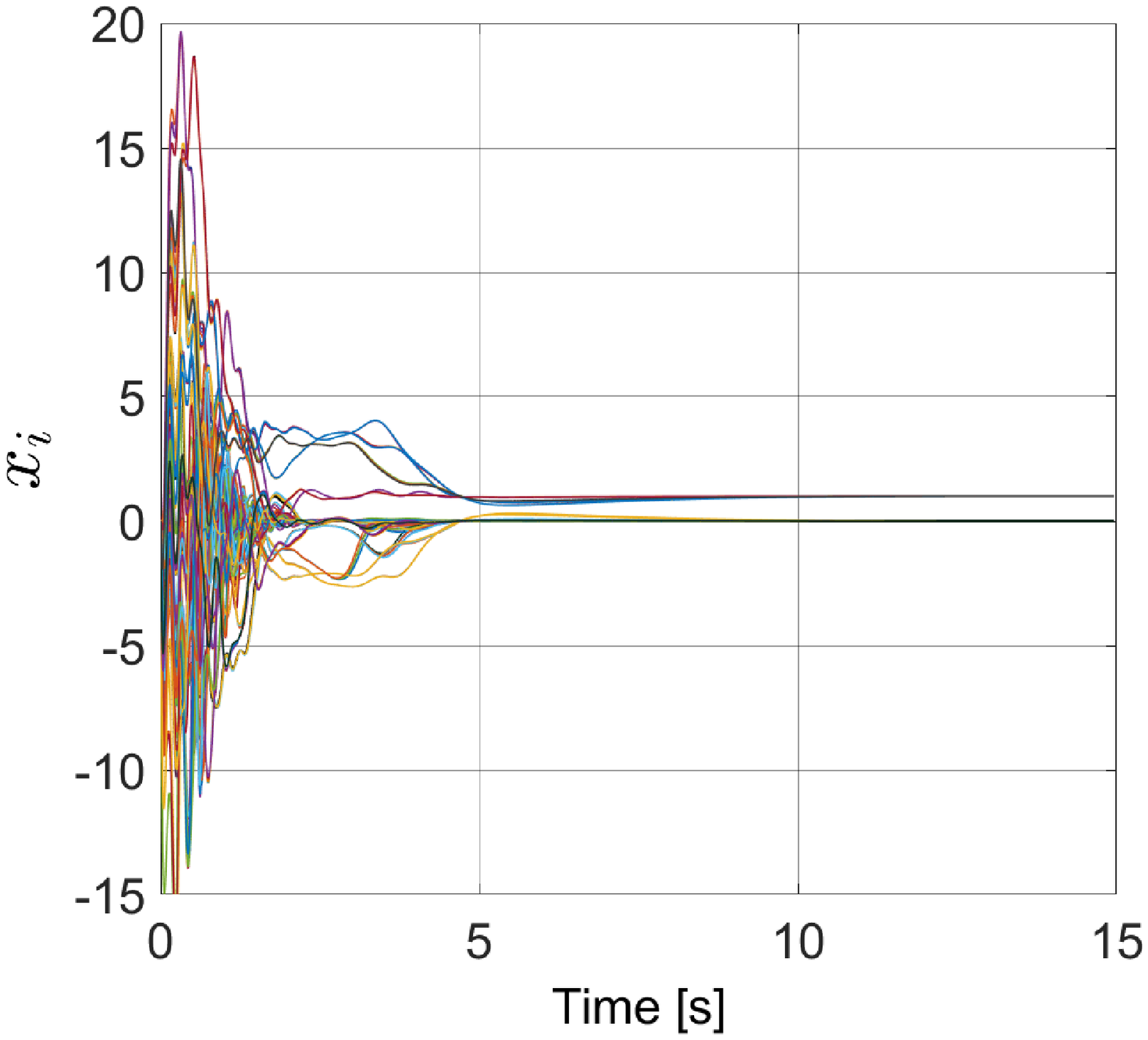}\label{fig: trajectories no delay}}
\subfigure[The algorithm \eqref{eq: overall distributed algorithm under delays} with scattering transformation under heterogeneous communication delays.]{\includegraphics[width = .48\linewidth]{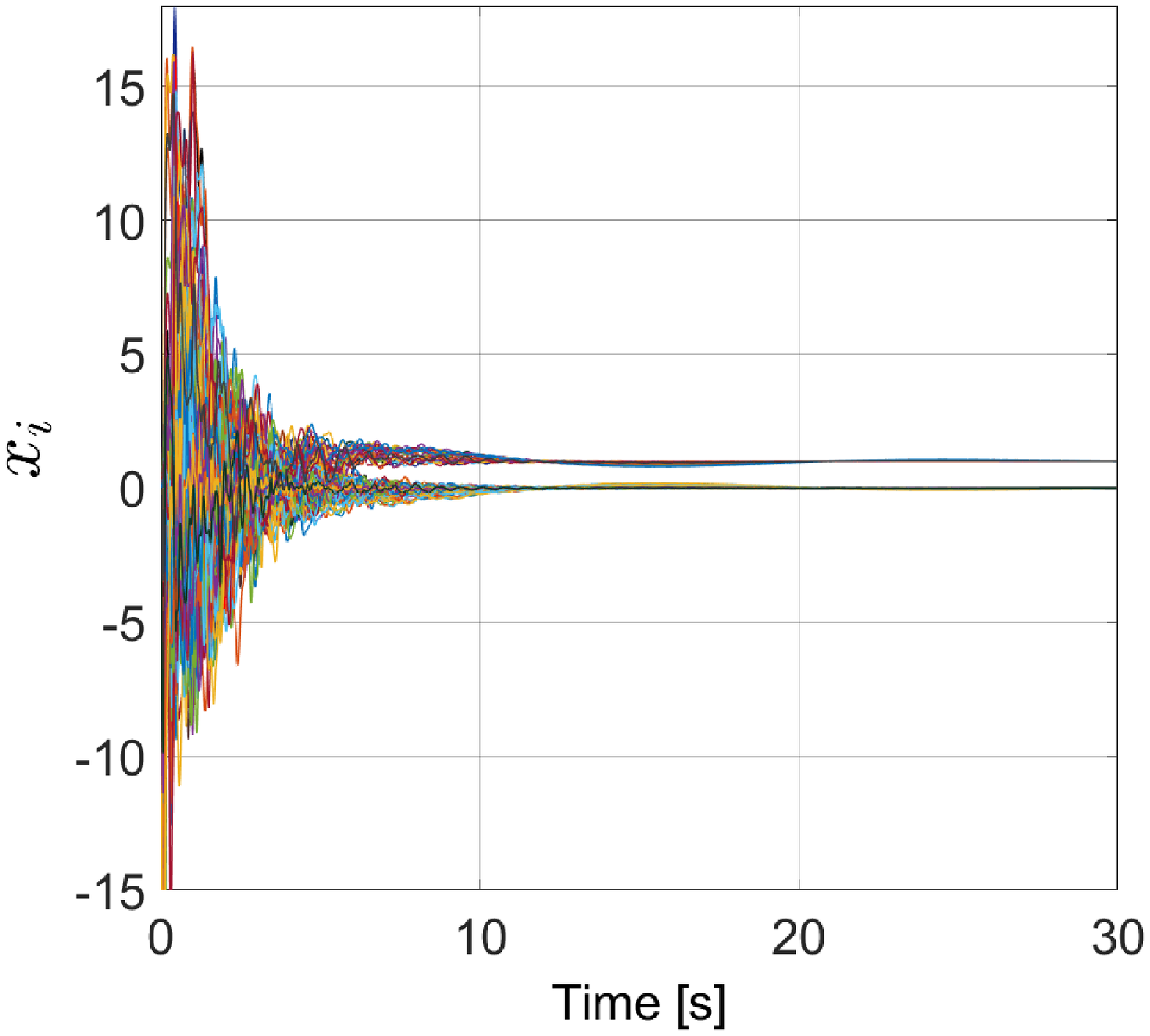}\label{fig: trajectories}}
\caption{The trajectories of $x_i$, $i = 1,\ldots,5$.}
\end{figure}
Note that the agents' states diverge for the present delays in the absence of the scattering transformation.
Therefore, the scattering transformation serves as a key technique for delay robustification.
The corresponding results are shown in \Cref{fig: Matching Results Robots and Targets}, illustrating the matching between Robots and Targets.


\section{Conclusion}
In this work, we have addressed the distributed constrained optimization problem under inter-agent communication delays from the perspective of passivity.
A smooth continuous-time algorithm for distributed constrained optimization with general convex objective functions has been proposed. The scattering transformation has been incorporated into the proposed algorithm to enhance the robustness against unknown and heterogeneous communication delays.

\bibliographystyle{IEEEtran}
\bibliography{References}
\end{document}